\documentclass[12pt]{amsart}
\usepackage{amsmath,latexsym,amscd,amsbsy,amssymb,amsfonts,amsthm,fleqn,leqno,
euscript, graphicx, texdraw, pb-diagram}
\usepackage[matrix,arrow,curve]{xy}

\numberwithin{equation}{section}
\newtheorem{thm}{Theorem}[section]
\newtheorem{pro}[thm]{Proposition}

\newtheorem{cor}[thm]{Corollary}

\theoremstyle{definition}

\theoremstyle{remark}

\newcommand{\bd}{\mathrm{bd}}

\numberwithin{equation}{section}

\begin{document}

\title[Local homological properties and cyclicity of homogeneous ANR compacta]
{Local homological properties and cyclicity of homogeneous ANR compacta}

\author{V. Valov}
\address{Department of Computer Science and Mathematics,
Nipissing University, 100 College Drive, P.O. Box 5002, North Bay,
ON, P1B 8L7, Canada} \email{veskov@nipissingu.ca}

\thanks{The author was partially supported by NSERC
Grant 261914-13.}


\begin{abstract}
In accordance with the Bing-Borsuk conjecture \cite{bb}, we show that if $X$ is an $n$-dimensional homogeneous metric $ANR$ compactum and $x\in X$, then there is a local basis at $x$ consisting of connected open sets $U$ such that the homological properties of $\overline U$ and $bd\,\overline U$ are similar to the
properties of the closed ball $\mathbb B^n\subset\mathbb R^n$ and its boundary $\mathbb S^{n-1}$. We discuss also the following questions raised by Bing-Borsuk \cite{bb}, where $X$ is a homogeneous $ANR$-compactum with $\dim X=n$:
\begin{itemize}
\item Is it true that $X$ is cyclic in dimension $n$?
\item Is it true that no non-empty closed subset of $X$, acyclic in dimension $n-1$, separates $X$?
\end{itemize}
It is shown that both questions have simultaneously positive or negative answers, and a positive solution to each one of them implies a solution to another question of Bing-Borsuk (whether every finite-dimensional  homogenous metric $AR$-compactum is a point).
\end{abstract}
\maketitle\markboth{}{Homogeneous $ANR$}





\section{Introduction}
There are few open problems concerning homogeneous compacta, see \cite{bb}. The most important one is the well known Bing-Borsuk conjecture stating that every $n$-dimensional homogeneous metric $ANR$-compactum $X$ is an $n$-manifold. Another one is whether any such $X$ has the following properties: $(i)$ $X$ is cyclic in dimension $n$; $(ii)$ no closed non-empty subset of $X$, acyclic in dimension $n-1$, separates $X$. It is also unknown if there exists a non-trivial finite dimensional metric homogeneous $AR$ compactum.

In this paper we address the above problems and investigate the homological structure of
homogeneous metric $ANR$-compacta. In accordance with the Bing-Borsuk conjecture, we prove that any such a compactum has local homological properties similar to the local structure of $\mathbb R^n$, see Theorem 1.1. It is also shown that the properties $(i)$ and $(ii)$ from the second of the above questions are equivalent, so each one of them implies that every finite-dimensional homogeneous metric $AR$ is a point.

Reduced \v{C}ech homology $H_n(X;G)$ and cohomology groups $H^n(X;G)$ with coefficients from $G$ are considered everywhere below, where $G$ is an abelian group. 
Suppose $(K,A)$ is a pair of closed subsets of a space $X$ with $A\subset K$. By $i^n_{A,K}$ we denote the homomorphism from $H_n(A;G)$ into $H_n(K;G)$ generated by the inclusion $A\hookrightarrow K$.
Following \cite{bb}, we say that
$K$ is an {\em $n$-homology membrane spanned on $A$ for an element $\gamma\in H_n(A;G)$} provided $\gamma$ is homologous to
zero in $K$, but not homologous to zero in any proper closed subset of $K$ containing $A$. It is well known \cite[Property 5., p.103]{bb} that for every compact metric space $X$ and a closed set $A\subset X$ the existence a non-trivial element $\gamma\in H_n(A;G)$ with $i^n_{A,X}(\gamma)=0$ yields the existence of a closed set $K\subset X$ containing $A$ such that $K$ is an $n$-homology membrane for $\gamma$ spanned on $A$.
We also say that a space $K$ is a homological {\em $(n,G)$-bubble} if $H_n(K;G)\neq 0$, but $H_n(B;G)=0$ for every closed proper subset $B\subset K$.

For any abelian group $G$, Alexandroff \cite{a} introduced the dimension $d_GX$ of a space $X$ as the maximum integer $n$ such that  there exist a closed set $F\subset X$ and a nontrivial element $\gamma\in H_{n-1}(F;G)$ with $\gamma$ being $G$-homologous to zero in $X$. According to \cite[p.207]{a} we have the following inequalities for any metric finite-dimensional compactum $X$: $d_GX\leq\dim X$ and $\dim X=d_{\mathbb Q_1}X=d_{\mathbb S^1}X$, where $G$ is any abelian group, $\mathbb S^1$ is the circle group and $\mathbb Q_1$ is the group of rational elements of $\mathbb S^1$.

Because the definition of $d_GX$ does not provide any information for the homology groups $H_{k-1}(F;G)$ when $F\subset X$ is closed and $k<d_GX-1$, we consider
the set $\mathcal{H}_{X,G}$ of all integers $k\geq 1$ such that there exist a closed set $F\subset X$ and a nontrivial element $\gamma\in H_{k-1}(F;G)$ with $i^{k-1}_{F,X}(\gamma)=0$. Obviously, $d_GX=\max\mathcal{H}_{X,G}$.

Using the properties of the sets $\mathcal{H}_{X,G}$, we investigate in Section 2 the local homological properties of metric homogeneous $ANR$ compacta. The main result in that section is Theorem 1.1 below, which is a homological version of \cite[Theorem 1.1]{vv1}.
\begin{thm}
Let $X$ be a finite dimensional homogeneous metric $ANR$ with $\dim X\geq 2$. Then every point $x\in X$ has a basis $\mathcal B_x=\{U_k\}$ of open sets such that for any abelian group $G$ and $n\geq 2$ with $n\in\mathcal{H}_{X,G}$ and $n+1\not\in\mathcal{H}_{X,G}$ almost all $U_k$ satisfy the following conditions:
\begin{itemize}
\item[(1)] $H_{n-1}(\bd\,\overline  U_k;G)\neq 0$ and $\overline{U}_k$ is an $(n-1)$-homology membrane spanned on $\bd\,\overline  U_k$ for any non-zero $\gamma\in H_{n-1}(\bd\,\overline  U_k;G)$;
\item[(2)] $H_{n-1}(\overline U_k;G)=H_{n}(\overline U_k;G)=0$ and $X\setminus\overline U_k$ is connected;
\item[(3)] $\bd\,\overline  U_k$ is a homological $(n-1,G)$-bubble.
\end{itemize}
\end{thm}


\begin{cor}
Let $X$ be as in Theorem $1.1$. Then $X$ has the following property for any abelian group $G$ and $n\geq 2$ with $n\in\mathcal{H}_{X,G}$ and $n+1\not\in\mathcal{H}_{X,G}$: If a closed subset $K\subset X$ is an $(n-1)$-homology membrane spanned on $B$ for some closed set $B\subset X$ and $\gamma\in H_{n-1}(B;G)$, then $(K\setminus B)\cap\overline{X\setminus K}=\varnothing$.
\end{cor}

In Section 3 we show that the following two statements are equivalent, where $\mathcal H(n)$ is the class of all homogeneous metric $ANR$-compacta $X$ with $\dim X=n$:
\begin{itemize}
\item[(1)] For all $n\geq 1$ and $X\in\mathcal H(n)$ there is a group $G$ such that $H^{n}(X;G)\neq 0$ (resp., $H_{n}(X;G)\neq 0$);
\item[(2)] If $X\in\mathcal H(n)$ with $n\geq 1$, and $F\subset X$ is a closed separator of $X$ with $\dim F=n-1$, then there exists a group $G$ such that  $H^{n-1}(F;G)\neq 0$ (resp., $H_{n-1}(F;G)\neq 0$).
\end{itemize}
Therefore, we have the following result (see Corollary 3.3):

\begin{thm}
Suppose for all $n\geq 1$ and all $X\in\mathcal H(n)$ the following holds: For every closed separator $F$ of $X$ with $\dim F=n-1$ there exists a group $G$ such that either $H^{n-1}(F;G)\neq 0$ or $H_{n-1}(F;G)\neq 0$. Then there is no homogeneous metric $AR$-compactum $Y$ with $\dim Y<\infty$.
\end{thm}
\section{Local homological properties of homogeneous $ANR$-compacta}

We begin this section with the following analogue of Theorem 8.1 from \cite{bb}.
\begin{pro}
Let $X$ be a locally compact and homogeneous separable metrizable $ANR$ space. Suppose there is a pair $F\subset K$ of
compact proper subsets of $X$ such that $K$ is contractible in $X$ and $K$ is a homological membrane for some $\gamma\in H_{n-1}(F;G)$.
If $(K\setminus F)\cap\overline{X\setminus K}\neq\varnothing$, then there exists of a proper compact subset $P\subset X$ contractible in $X$ such that $H_{n}(P;G)\neq 0$.
\end{pro}

\begin{proof}
We follow the proof of \cite[Lemma 1]{mo} (let us note that the proof of Proposition 2.1 can also be obtained following the arguments of \cite[Theorem 8.1]{bb}). Let $a\in (K\setminus F)\cap\overline{X\setminus K}$. Then $a$ is a boundary point for $K$. Because $K$ is contractible in $X$, there is a homotopy $g:K\times [0,1]\to X$ such that $g(x,0)=x$ and $g(x,1)=c\in X$ for all $x\in K$. Then we can find an open set $U\subset X$ containing $K$ and a homotopy $\overline g:\overline U\times [0,1]\to X$ extending $g$ and connecting the identity on $\overline U$ and the constant map $\overline U\rightarrow c$ (this can be done since $X$ is an $ANR$). So, $\overline U$ is also contractible in $X$. Moreover, we can assume that $\overline U$ is compact. Fix a metric $d$ on $X$ generating its topology in the following way: consider $X$ as a subspace of its one-point compactification $\alpha X$ and take $d$ to be the restriction to $X$ of some admissible metric on $\alpha X$. Let $2\epsilon=d(a,F)$ and take an open cover $\omega$ of $U$ such that for any two $\omega$-close maps $f_1,f_2:K\to U$ (i.e., for all $x\in K$ the points $f_1(x),f_2(x)$ are contained in some element of $\omega$) there is an $\epsilon$-homotopy $\Phi:K\times [0,1]\to U$ between $f_1$ and $f_2$ (i.e., each set $M_\Phi(x)=\{\Phi(x,t):t\in [0,1]\}$, $x\in K$, is of diameter $<\epsilon$). This can be done because $U$ is an $ANR$. Now, we fix an open set $V\subset X$ containing $K$ with $\overline V\subset U$ and let $\delta$ be the Lebesgue's number of the open cover $\{\Gamma\cap\overline V:\Gamma\in\omega\}$ of $\overline V$. According to Effros' theorem \cite{e}, there is a positive
number $\eta$  such that if $x,y\in X$ are two points with $d(x,y)<\eta$, then $f(x)=y$ for some homeomorphism $h\colon X\to X$, which is $\min\{\delta,d(K,X\setminus V\}$-close to the identity on $X$ (the Effros' theorem can be applied because of the special choice of the metric $d$). Since $a$ is a boundary point for $K$, we can choose a point $b\in V\setminus K$ with $d(a,b)<\eta$. Then, there exists a homeomorphism $h_1':X\to X$ such that
$h_1'(a)=b$ and $d(x,h_1'(x))<\min\{\delta,d(K,X\setminus V\}$, $x\in X$. Let $h_1$ be the restriction $h_1'|K$. Obviously, $h_1:K\to h_1(K)$ is a homeomorphism with $h_1(K)\subset V$ and $h_1$ is $\delta$-close to the identity on $K$. Then, according to the choice of $\delta$, there is homotopy $h:K\times [0,1]\to U$ such that $h(x,0)=x$,
$h(x,1)=h_1(x)$ and $d(x,h(x,t))<\epsilon$ for all $x\in K$ and $t\in [0,1]$.

Let $K_1=K\cup h(F\times\mathbb I)$, $K_2=h_1(K)$ and $K_0=K_1\cap K_2$, where $\mathbb I=[0,1]$. Since $2\epsilon=d(a,F)$ and $h$ is an $\epsilon$-small homotopy, $b\in K_2\setminus K_1$. So, $K_0$ is a proper subset of $K_2$ containing $h_1(F)$. Hence, $D=h_1^{-1}(K_0)$ is a proper subset of $K$ containing $F$, which implies $\gamma_1=i^{n-1}_{F,D}(\gamma)\neq 0$. Because $h_1$ is a homeomorphism, $(\varphi_1)_*:H_{n-1}(D;G)\to H_{n-1}(K_0;G)$ is an isomorphism, where $\varphi_1=h_1|D$. Thus,
$\hat{\gamma}=(\varphi_1)_*(\gamma_1)\neq 0$.

\textit{Claim $1$. $i^{n-1}_{K_0,K_1}(\hat\gamma)=0$ and $i^{n-1}_{K_0,K_2}(\hat\gamma)=0$.}

Let $\lambda=i_{F,h(F\times\mathbb I)}^{n-1}(\gamma)$. Since $h|(F\times\mathbb I)$ is a homotopy between the identity on $F$ and the map $\varphi_2=h_1|F$, $\lambda=i_{h_1(F),h(F\times\mathbb I)}^{n-1}((\varphi_2)_*(\gamma))$. We consider the following commutative diagram

$$
\xymatrix{
&\gamma\in H_{n-1}(F;G)\ar[rr]^{i_{F,D}^{n-1}}\ar[d]_{(\varphi_2)_*}&& H_{n-1}(D;G)\ni\gamma_1\ar[d]_{(\varphi_1)_*}\\
&(\varphi_2)_*(\gamma)\in H_{n-1}(h_1(F);G)\ar[rr]^{i_{h_1(F),K_0}^{n-1}}\ar[d]_{i_{h_1(F),h(F\times\mathbb I)}^{n-1}}&& H_{n-1}(K_0;G)\ni\hat\gamma\ar[ddll]_{i_{K_0,K_1}^{n-1}}\\
&\lambda\in H_{n-1}(h(F\times\mathbb I);G)\ar[d]_{i_{h(F\times\mathbb I),K_1}^{n-1}}\\
&0\in H_{n-1}(K_1;G)\\
}
$$
Obviously, $i_{F,K_1}^{n-1}(\gamma)=i_{h(F\times\mathbb I),K_1}^{n-1}(\lambda)=i_{K_0,K_1}^{n-1}(\hat\gamma)$. On the other hand, $i_{F,K_1}^{n-1}(\gamma)=i_{K,K_1}^{n-1}(i_{F,K}^{n-1}(\gamma))=0$ because $i_{F,K}^{n-1}(\gamma)=0$. Hence, $i_{K_0,K_1}^{n-1}(\hat\gamma)=0$.

For the second part of the claim, observe that $i_{D,K}^{n-1}(\gamma_1)=i_{F,K}^{n-1}(\gamma)=0$.
Then, the equality $i^{n-1}_{K_0,K_2}(\hat\gamma)=0$ follows from next diagram
$$
\xymatrix{
&\gamma_1\in H_{n-1}(D;G)\ar[rr]^{i_{D,K}^{n-1}}\ar[d]_{(\varphi_1)_*}&& H_{n-1}(K;G)\ni 0\ar[d]_{(h_1)_*}\\
&\hat\gamma\in H_{n-1}(K_0;G)\ar[rr]^{i_{K_0,K_2}^{n-1}}&& H_{n-1}(K_2;G)\ni 0\\
}
$$

We are in a position now to complete the proof of Proposition 3.1. Let $P=K_1\cup K_2$. Since $h(K\times\mathbb I)\subset U$, $P\subset U$. Therefore, $P$  is contractible in $X$ (recall that $\overline U$ is contractible in $X$). Finally, by Claim 1 and the Phragmen-Brouwer theorem (see \cite{bb}),
there exists a non-trivial $\alpha\in H_{n}(P;G)$. 
\end{proof}

For simplicity, we say that a closed set $F\subset X$ is {\em strongly contractible in $X$} if $F$ is contractible in a closed set $A\subset X$ and
$A$ is contractible in $X$.
\begin{cor}
Let $X$ be a homogeneous compact metrizable $ANR$-space such that $n\in\mathcal{H}_{X,G}$ and $n+1\not\in\mathcal{H}_{X,G}$. Then for every closed set $F\subset X$ we have:
\begin{itemize}
\item[(1)] $H_{n}(F;G)=0$ provided $F$ is contractible in $X$;
\item[(2)] $F$ separates $X$ provided $H_{n-1}(F;G)\neq 0$ and $F$ is strongly contractible in $X$;
\item[(3)] If $K$ is a homological membrane for some non-trivial element of $H_{n-1}(F;G)$ and $K$ is contractible in $X$, then $(K\setminus F)\cap\overline{X\setminus K}=\varnothing$.
\end{itemize}
\end{cor}

\begin{proof}
Since $F$ is contractible in $X$, every $\gamma\in H_{n}(F;G)$ is homologous to zero in $X$. So, the existence of a non-trivial element of
$H_{n}(F;G)$ would imply $n+1\in\mathcal{H}_{X,G}$, a contradiction.

To prove the second item, suppose $H_{n-1}(F;G)\neq 0$ and $F$ is strongly contractible in $X$. So, there exists a closed set $A\subset X$ such that $F$ is contractible in $A$ and $A$ is contractible in $X$. Then, by \cite[Property 5., p.103]{bb}, we can find a closed set $K\subset A$ containing $F$ which is a homological membrane for some non-trivial $\gamma\in H_{n-1}(F;G)$. Because $K$ (as a subset of $A$) is contractible in $X$, the assumption $(K\setminus F)\cap\overline{X\setminus K}\neq\varnothing$ would yield the existence of a proper closed set $P\subset X$ contractible in $X$ with $H_{n}(P;G)\neq 0$ (see Proposition 2.1). Consequently, there would be a non-trivial $\alpha\in H_{n}(P;G)$ homologous to zero in $X$. Hence, $n+1\in\mathcal{H}_{X,G}$, a contradiction. Therefore, $(K\setminus F)\cap\overline{X\setminus K}=\varnothing$. This means that
$X\setminus F=(K\setminus F)\cup (X\setminus K)$ with both $K\setminus F$ and $X\setminus K$ being  non-empty open disjoint subsets of $X$.

The above arguments provide also the proof of the third item.
\end{proof}

\textit{Proof of Theorem $1.1$.}
Suppose $X$ satisfies the hypotheses of Theorem 1.1. By \cite[Theorem 1.1]{vv1}, every $x\in X$ has a basis $\mathcal B_x=\{U_k\}_{k\geq 1}$ of open sets
satisfying the following conditions:
 $\bd U_k=\bd\,\overline  U_k$; the sets $U_k$, $\bd\,\overline  U_k$ and $X\setminus\overline U_k$ are connected; $H^{\dim X-1}(A;\mathbb Z)=0$ for all proper closed sets $A\subset\bd\,\overline  U_k$. We may also suppose that each $\overline U_{k+1}$ is contractible in $U_k$ and all $\overline U_k$ are strongly contractible in $X$. Let $G$ be an abelian group and $n\geq 2$ with $n+1\not\in\mathcal{H}_{X,G}$ and $n\in\mathcal{H}_{X,G}$.
So, there exist a closed set $B\subset X$ and a nontrivial element $\gamma\in H_{n-1}(B;G)$ with $i^{n-1}_{B,X}(\gamma)=0$. Then, by \cite[Property 5., p.103]{bb},
there is a closed set $K\subset X$ containing $B$ which is a homological membrane for $\gamma$.
 We fix a point $\widetilde{x}\in K\setminus B$
and its open in $K$ neighborhood $W$ with $\overline W\cap B=\varnothing$. According to \cite[Property 6., p.103]{bb}, $\overline W$ is an $(n-1)$-homological membrane for some non-trivial element of $H_{n-1}(\bd_K\overline W;G)$. We can choose $W$ so small that $\overline W$ is contractible in $X$. Then Corollary 2.2 yields $(\overline W\setminus \bd_K\overline W)\cap\overline{X\setminus\overline W}=\varnothing$. So, $\overline W\setminus \bd_K\overline W$ is open in $X$ and contains $\widetilde x$. Hence, there exists $k_0$ such that $U_k\subset \overline W\setminus \bd_K\overline W$ for all $U_k\in\mathcal B_{\widetilde x}$ with $k\geq k_0$.
Below we consider only the elements $U_k$ with $k\geq k_0$. Applying again \cite[Property 6., p.103]{bb}, we conclude that every $\overline U_k$ is a homological membrane for some non-trivial element of $H_{n-1}(\bd U_k;G)$.  By Corollary 2.2(1), $H_n(\overline U_k;G)=0$.
Suppose $\gamma\in H_{n-1}(\bd U_k;G)$ is non-trivial. Since $X\setminus\overline U_k$ is connected,  Corollary 2.2(2) implies that
$H_{n-1}(\overline U_k;G)=0$. Consequently, $\gamma$ is homologous to zero in
$\overline U_k$. So, by \cite[Property 5., p.103]{bb}, $\overline U_k$ contains a closed set $P$ such that $P$ is a homological membrane for $\gamma$. Then Corollary 2.2(3) implies
$(P\setminus \bd U_k)\cap\overline{X\setminus P}=\varnothing$. Hence, $X\setminus \bd U_k$ is the union of the disjoint open sets $P\setminus \bd U_k$
and $X\setminus P$. Because $U_k$ is connected and $U_k\cap P\neq\varnothing$, $U_k\subset P\setminus \bd U_k$. Therefore, $P=\overline U_k$.
This provides the proof of the first two conditions of Theorem 1.1.

To prove the last item of Theorem 1.1, assume that $H_{n-1}(F;G)\neq 0$ for some closed proper subset $F$ of $\bd U_{k+1}$, where $k\geq k_0$. Because $F$ (as a subset of $\overline U_{k+1}$) is strongly contractible in $X$, according to Corollary 2.2(2), $F$ separates $X$.
So, $X\setminus F$ is the union of two disjoint non-empty open in $X$ sets $V_1$ and $V_2$ with $\overline V_1\cap\overline V_2\subset F$. Let us show that $F$ separates $\overline U_{k}$. Indeed, otherwise
$\overline U_{k}\setminus F$ would be connected. Then $\overline U_{k}\setminus F$ should be contained in one of the sets $V_1, V_2$, say $V_1$. Since
$X\setminus\overline U_{k}$ is also connected and $V_2\neq\varnothing$, $X\setminus\overline U_{k}\subset V_2$. Hence,
$\overline{\overline U_{k}\setminus F}\cap\overline{X\setminus\overline U_{k}}\subset F$.
On the other hand, because $\overline U_{k}\setminus F$ is dense in $\overline U_{k}$ (recall that $F$ does not contain interior points),
$\overline{\overline U_{k}\setminus F}\cap\overline{X\setminus\overline U_{k}}=\bd U_{k}$. So, $F\supset bdU_{k}$, a contradiction. Therefore, $F$ separates $\overline U_{k}$.

The proof of Theorem 1.1(3) will be done if we show that $F$ can not separate $\overline U_k$. According to \cite[Theorem 1.1]{vv1},
$\overline{U}_k$ is an $(m-1)$-cohomology membrane spanned on $\bd U_k$ for some non-trivial $\alpha\in H^{m-1}(\bd U_k;\mathbb Z)$, where $m=\dim X$. This means that $\alpha$ (considered as a map from $\bd U_k$ to the Eilenberg-MacLane complex $K(\mathbb Z,m-1)$) is not extendable over $\overline{U}_k$ but it is extendable over any proper closed subset of $\overline{U}_k$. Hence, by \cite[Proposition 2.10]{vt}, the couple $(\overline{U}_k,\bd U_k)$ is a strong $K_\mathbb Z^{m}$-manifold (see \cite{vt} for the definition of a strong $K_\mathbb Z^m$-manifold). So, according to \cite[Theorem 3.3]{vt}, $H^{m-1}(F;\mathbb Z)\neq 0$ because $F$ separates $\overline U_k$ and $F\cap \bd U_k=\varnothing$. Finally, we obtained a contradiction because $H^{m-1}(A;\mathbb Z)=0$ for every proper closed subset $A$ of $\bd U_{k+1}$. Therefore, the all $U_k$, $k\geq k_0+1$, satisfy conditions $(1) - (3)$ from Theorem 1.1. $\Box$

\textit{Proof of Corollary $1.2$.} Suppose there exists a point $a\in (K\setminus B)\cap\overline{X\setminus K}$ and take a set $U\in\mathcal B_a$ satisfying
condition $(1)-(3)$ from Theorem 1.1 such that $U\cap B=\varnothing$. Then $F_U=\bd_K(U\cap K)$ is non-empty and
it follows from \cite[Property 6., p.103]{bb} that
  $\overline{U\cap K}$ is a homology membrane for some non-zero $\alpha\in H_{n-1}(F_U;G)$. Because $F_U\subset\bd U$, by Theorem 1.1(3), $F_U=\bd U$. So, $\overline U$ is a homological membrane for $\alpha$, see Theorem 1.1(1). This implies that $i^{n-1}_{\bd U,\overline{U\cap K}}(\alpha)\neq 0$ provided $\overline{U\cap K}$ is a proper subset of $\overline U$. Therefore,
$\overline{U\cap K}=\overline U$, which yields $U\subset K$. The last inclusion contradicts the fact that $a\in\overline{X\setminus K}$. Hence,
$(K\setminus B)\cap\overline{X\setminus K}=\varnothing$. $\Box$

\section{Cyclicity of homogeneous $ANR$'s}

Let $\mathcal H(n)$ be the class of all homogeneous metric $ANR$-compacta $X$ with $\dim X=n$.

\begin{thm}
The following conditions are equivalent:
\begin{itemize}
\item[(1)] If $n\geq 1$, then for every space $X\in\mathcal H(n)$ there exists a group $G$ with $H^n(X;G)\neq 0$;
\item[(2)] If $n\geq 1$ and $X\in\mathcal H(n)$, then for every closed set $F\subset X$ separating $X$ there exists a group $G$ with $H^{n-1}(F;G)\neq 0$;
\item[(3)] If $n\geq 1$ and $X\in\mathcal H(n)$, then for every $(n-1)$-dimensional closed set $F\subset X$ separating $X$ there exists a group $G$ with $H^{n-1}(F;G)\neq 0$.
\end{itemize}
\end{thm}

\begin{proof}
$(1)\Rightarrow (2)$ Suppose $n\geq 1$ and $X\in\mathcal H(n)$. Then $H^n(X;G)\neq 0$ for some group $G$, and by \cite[Corollary 1.2]{vv}, $H^{n-1}(F;G)\neq 0$ for every non-empty closed set $F\subset X$ separating $X$.

$(2)\Rightarrow (3)$ This implication is trivial.

$(3)\Rightarrow (1)$ Suppose that condition $(3)$ holds, but there exists $n\geq 1$ and $X\in\mathcal H(n)$ such that $H^n(X;G)=0$ for all groups $G$.
Consider the two-dimensional sphere
$\mathbb S^2$ and a circle $\mathbb S^1$ separating $\mathbb S^2$. Then $X\times\mathbb S^2\in\mathcal H(n+2)$ and $X\times\mathbb S^1$ is a closed separator of $X\times\mathbb S^2$ of dimension $n+1$. So, there is a group $G'$ such that $H^{n+1}(X\times\mathbb S^1;G')\neq 0$. On the other hand, according to the K\"{u}nneth formula, we have the exact sequence
$$\sum_{i+j=n+1}H^i(X)\otimes H^j(\mathbb S^1)\to H^{n+1}(X\times\mathbb S^1)\to\sum_{i+j=n+2}H^i(X)*H^j(\mathbb S^1),$$
where the coefficient group $G'$ is suppressed. Because $\dim X=n$ and $\dim\mathbb S^1=1$, $H^{n+i}(X;G')=0$ and $H^{1+i}(\mathbb S^1;G')=0$ for all $i\geq 1$.
Moreover, $H^n(X;G')=0$. So, $$\sum_{i+j=n+1}H^i(X;G')\otimes H^j(\mathbb S^1;G')=\sum_{i+j=n+2}H^i(X;G')*H^j(\mathbb S^1;G')=0.$$
Hence, $H^{n+1}(X\times\mathbb S^1;G')=0$, a contradiction.
\end{proof}

A homological version of Theorem 3.1 also holds.
\begin{thm}
The following conditions are equivalent:
\begin{itemize}
\item[(1)] If $n\geq 1$, then for every space $X\in\mathcal H(n)$ there exists a group $G$ with $H_n(X;G)\neq 0$;
\item[(2)] If $n\geq 1$ and $X\in\mathcal H(n)$, then for every closed set $F\subset X$ separating $X$ there exists a group $G$ with $H_{n-1}(F;G)\neq 0$;
\item[(3)] If $n\geq 1$ and $X\in\mathcal H(n)$, then for every $(n-1)$-dimensional closed set $F\subset X$ separating $X$ there exists a group $G$ with $H_{n-1}(F;G)\neq 0$.
\end{itemize}
\end{thm}

\begin{proof}
Everywhere below, $\widehat{H}_*$ denotes the exact homology (see \cite{ma}, \cite{sk}), which for locally compact metric spaces is equivalent to Steenrod's homology \cite{st}. For every compact metric space $X$ and every $k$ there exists a surjective homomorphism
$T^k_{X}:\widehat{H}_{k}(X;G)\to H_{k}(X;G)$.
According to \cite[Theorem 4 and p.125]{sk}, $T^k_{X}$ is an isomorphism in each of the following cases: $G$ is a vector space over a field; both $\widehat{H}_{k}(X;G)$ and $G$ are countable modules; $\dim X=k$; $H^{k+1}(X;\mathbb Z)$ is finitely generated.

$(1)\Rightarrow (2)$ Suppose $n\geq 1$ and $X\in\mathcal H(n)$. Then $H_n(X;G)\neq 0$ for some group $G$.
 By \cite[Theorem 3]{sk}, we have the exact sequence
$$\mathrm{Ext}(H^{n+1}(X;\mathbb Z),G)\to\widehat{H}_{n}(X;G)\to\mathrm{Hom}(H^{n}(X;\mathbb Z),G)\to 0.\leqno{(*)}$$
Since $\dim X=n$, $H^{n+1}(X;\mathbb Z)=0$. Moreover $\widehat{H}_{n}(X;G)$ is non-trivial because so is $H_{n}(X;G)$ and $T^n_X$ is a surjective homomorphism. Hence, $H^{n}(X;\mathbb Z)\neq 0$ and there exists a non-trivial homomorphism $\varphi\colon H^{n}(X;\mathbb Z)\to G$. Now, let $F\subset X$ be a closed separator of $X$ and $X\setminus F=X_1\cup X_2$, where $X_1,X_2\subset X$ are closed proper subsets with $X_1\cap X_2=F$. Since $H^n(P;\mathbb Z)=0$ for every closed proper subset $P\subset X$ (see \cite{vv}), $H^n(F;\mathbb Z)=H^n(X_1;\mathbb Z)=H^n(X_2;\mathbb Z)=0$. Then it follows from the Mayer-Vietoris sequence
$$
\xymatrix{
H^{n-1}(F;\mathbb Z)\ar[rr]^{\partial} && H^{n}(X;\mathbb Z)\ar[rr]^{\psi}&&H^{n}(X_1;\mathbb Z)\oplus H^{n}(X_1;\mathbb Z)\\
}
$$
that $H^{n-1}(F;\mathbb Z)\neq 0$ and $\partial$ is a surjective homomorphism. Consequently, $\varphi\circ\partial:H^{n-1}(F;\mathbb Z)\to G$ is also a non-trivial surjective homomorphism. Hence, $\mathrm{Hom}(H^{n-1}(F;\mathbb Z),G)\neq 0$ and the exact sequence
$$0\to\mathrm{Ext}(H^{n}(F;\mathbb Z),G)\to\widehat{H}_{n-1}(F;G)\to\mathrm{Hom}(H^{n-1}(F;\mathbb Z),G)\to 0$$ yields $\widehat{H}_{n-1}(F;G)\neq 0$.
Finally, since $H^n(F;\mathbb Z)=0$, $\widehat{H}_{n-1}(F;G)$ is isomorphic to $H_{n-1}(F;G)$.

$(2)\Rightarrow (3)$ This implication is obvious.

$(3)\Rightarrow (1)$ As in the proof of Theorem 3.1, $(3)\Rightarrow (1)$, suppose there exists $n\geq 1$ and $X\in\mathcal H(n)$ such that $H_n(X;G)=0$ for all groups $G$. Since $\widehat{H}_n(X;G)$ is isomorphic to $H_n(X;G)$ and $H^{n+1}(X;\mathbb Z)=0$
(recall that $\dim X=n$), it follows from the exact sequence $(*)$ that
$\mathrm{Hom}(H^{n}(X;\mathbb Z),G)=0$ for all groups $G$. This implies $H^{n}(X;\mathbb Z)=0$. As above, the product
$X\times\mathbb S^1$ is a closed separator of $X\times\mathbb S^2$, and according to our assumption, $H_{n+1}(X\times\mathbb S^1;G')\neq 0$ for some group $G'$. Because
$\dim X\times\mathbb S^1=n+1$, $H_{n+1}(X\times\mathbb S^1;G')\cong\widehat{H}_{n+1}(X\times\mathbb S^1;G')$ and $H^{n+2}(\dim X\times\mathbb S^1,\mathbb Z)=0$. Therefore, the exact sequence
$$\mathrm{Ext}(H^{n+2}(X\times\mathbb S^1),G')\to\widehat{H}_{n+1}(X\times\mathbb S^1;G')\to\mathrm{Hom}(H^{n+1}(X\times\mathbb S^1),G'),$$ where the coefficient group $\mathbb Z$ in $H^{n+2}(X\times\mathbb S^1)$ and $H^{n+1}(X\times\mathbb S^1)$ are suppressed, yields
that $H^{n+1}(X\times\mathbb S^1;\mathbb Z)\neq 0$. On the other hand,
the K\"{u}nneth formula from the proof of Theorem 3.1 (with $\mathbb Z$ being the coefficient group in all cohomology groups) implies $H^{n+1}(X\times\mathbb S^1;\mathbb Z)=0$, a contradiction.
\end{proof}

\begin{cor}
Suppose for all $n\geq 1$ and all $X\in\mathcal H(n)$ the following holds: For every closed separator $F$ of $X$ with $\dim F=n-1$ there exists a group $G$ such that either $H^{n-1}(F;G)\neq 0$ or $H_{n-1}(F;G)\neq 0$. Then there is no homogeneous metric $AR$-compactum $Y$ with $\dim Y<\infty$.
\end{cor}

If $\mathcal H(G,n)$ denotes the class of all homogeneous metric $ANR$-compacta $X$ with $\dim_GX=n$, the arguments from Theorem 3.1  provide the following result:

\begin{pro}
The following conditions are equivalent:
\begin{itemize}
\item[(1)] $H^n(X;G)\neq 0$ for all $X\in\mathcal H(G,n)$ and all $n\geq 1$;
\item[(2)] If $X\in\mathcal H(G,n)$ and $n\geq 1$, then $H^{n-1}(F;G)\neq 0$ for every closed set $F\subset X$ separating $X$;
\item[(3)] If $X\in\mathcal H(G,n)$ and $n\geq 1$, $H^{n-1}(F;G)\neq 0$ for every closed set $F\subset X$ separating $X$ with $\dim_GF=n-1$.
\end{itemize}
\end{pro}
The corresponding homological analogue of Proposition 3.4 also holds for some groups $G$.
\begin{pro}
The following conditions are equivalent, where $G$ is either a field or a torsion free group:
\begin{itemize}
\item[(1)] $H_n(X;G)\neq 0$ for all $X\in\mathcal H(n)$ and all $n\geq 1$;
\item[(2)] If $X\in\mathcal H(n)$, $n\geq 1$, and $F\subset X$ is a closed set separating $X$, then $H_{n-1}(F;G)\neq 0$;
\item[(3)] If $X\in\mathcal H(n)$, $n\geq 1$, and $F\subset X$ is a closed set separating $X$ with $\dim F=n-1$, then $H_{n-1}(F;G)\neq 0$.
\end{itemize}
\end{pro}

\begin{proof}
All implications except $(3)\Rightarrow (1)$ follow from the proof of Theorem 3.2. To prove $(3)\Rightarrow (1)$, we suppose there exists a space $X\in\mathcal H(n)$ with $H_n(X;G)=0$. Considering the $(n+1)$-dimensional separator $X\times\mathbb S^1$ of $X\times\mathbb S^2$, we conclude that $H_{n+1}(X\times\mathbb S^1;G)\neq 0$.
Because $X$ and $X\times\mathbb S^1$ are $ANR$'s, their \v{C}ech homology groups are isomorphic to the singular homology groups. Thus, we can apply the
K\"{u}nneth formula
$$\sum_{i+j=n+1}H_i(X)\otimes H_j(\mathbb S^1)\to H_{n+1}(X\times\mathbb S^1)\to\sum_{i+j=n}H_i(X)*H_j(\mathbb S^1),$$
where $G$ is the coefficient group. Since $H_n(X;G)=H_{n+1}(X;G)=0$ and $H_j(\mathbb S^1;G)=0$ for all $j>1$, $\sum_{i+j=n+1}H_i(X;G)\otimes H_j(\mathbb S^1;G)=0$.
If $G$ is a torsion free group, then the group $\sum_{i+j=n}H_i(X;G)*H_j(\mathbb S^1;G)$ is also trivial because $H_1(\mathbb S^1;G)=G$ yields
$H_{n-1}(X;G)*H_1(\mathbb S^1;G)=0$. Therefore, $H_{n+1}(X\times\mathbb S^1;G)=0$, a contradiction.

When $G$ is a field, the group $H_{n+1}(X\times\mathbb S^1;G)$ is isomorphic to the trivial group $\sum_{i+j=n+1}H_i(X;G)\otimes H_j(\mathbb S^1;G)$.
So, we have again a contradiction.
\end{proof}



\end{document}